\documentclass{proc-l}

\usepackage{amscd}

\copyrightinfo{2009}{American Mathematical Society}
\newtheorem{theorem}{Theorem}[section]
\newtheorem{lemma}[theorem]{Lemma}
\newtheorem{coro}[theorem]{Corollary}
\theoremstyle{definition}
\newtheorem{definition}[theorem]{Definition}
\newtheorem{example}[theorem]{Example}

\theoremstyle{remark}

\numberwithin{equation}{section}

\begin{document}

\title[Bifurcation in random dynamical systems]{A sufficient condition for bifurcation in random dynamical systems}

\author[X. Chen]{Xiaopeng Chen}

\address{School of Mathematics and Statistics\\ Huazhong University of Science and
Technology\\  Wuhan 430074, China}

\email{chenxiao002214336@yahoo.cn}

\thanks{The first author would like to thank  Zhenxin Liu for helpful comments.}

\author[J. Duan]{Jinqiao Duan}
\address{Department of Applied Mathematics\\ Illinois Institute of Technology\\ Chicago, IL 60616, USA}
\email{duan@iit.edu}
\thanks{The second author was supported in part by NSF Grant 0620539, the Cheung Kong Scholars Program and the
K. C. Wong Education Foundation.}
\author[X. Fu]{Xinchu Fu}
\address{
Department of Mathematics\\ Shanghai University\\ Shanghai 200444,
China} \email{xcfu@shu.edu.cn}
\thanks{The third author was supported in part by NSFC Grant 10672146.}
\subjclass[2000]{Primary 37H20, 37B30, 60H10.}

\date{April 19 , 2009 and, in revised form, June 24, 2009.}
\commby{Yingfei Yi}

\keywords{Random dynamical systems, discrete-time  and
continuous-time dynamical systems, random homeomorphism, Conley
index, abstract bifurcation point}

\begin{abstract}
Some properties of random Conley index are obtained and then a
sufficient condition for the existence of abstract bifurcation
points for both discrete-time and continuous-time random dynamical
systems is presented. This stochastic bifurcation phenomenon is
demonstrated by a few examples.
\end{abstract}

\maketitle

\section{Introduction}
The Conley index   is a  topological tool for  investigating
dynamical systems  \cite{6,Krz,Mischaikow,Smo}. In particular, the
Conley index has been  used to detect   bifurcation   in
deterministic dynamical systems \cite{ Bar,Fu,Izy, Kap, Kap1, Kun}.
Recently, the Conley index is defined for discrete-time random
dynamical systems \cite{Liu}. In the present paper, we
 investigate     bifurcation
  for both discrete-time and continuous-time random dynamical
systems, via random Conley index.

 We first   present a
 few properties of random Conley index in \S 2.  Then in
\S 3, we introduce a concept of prime random isolated invariant
sets. Finally in \S 4, we obtain a sufficient condition for the
existence of abstract bifurcation points for random dynamical
systems.

\section{Some properties of random Conley index}

   A continuous random dynamical system (RDS) in the state space $X$,
    with the   time set $\mathbb{T}$ and the underlying probability space $(\Omega, \mathcal{F}, \mathbb{P})$,
 consists of two ingredients \cite{Arn}: \\
 \\
(i) A model of the noise, namely a driving   flow $(\theta_t)_{t\in
\mathbb{T}}$ on the sample space $\Omega$, such that
$(t,\omega)\mapsto \theta_t \omega$ is a measurable flow that leaves
$\mathbb{P}$ invariant, i.e. $\theta_t \mathbb{P} =
\mathbb{P}$ for all $t \in \mathbb{T}$. \\
(ii) A model of the system evolution influenced by noise, namely a
cocycle $\phi$ over $\theta$, i.e., a measurable mapping $\phi:
\mathbb{T} \times\Omega\times X\rightarrow  X, $
$(t,\omega,x)\mapsto \phi(t,\omega,x)$, such that $(t,x)\mapsto
\phi(t,\omega,x)$ is continuous for all $\omega \in \Omega$ and the
family $\phi(t,\omega,\cdot) = \phi(t,\omega) : X \rightarrow X$ of
random mappings  satisfies the cocycle property:
\begin{eqnarray}\qquad
\phi(0,\omega)=id_X , \; \; \phi(t+s,\omega)=\phi(t,\theta_s
\omega)\circ \phi(s,\omega) \; \mbox{ for all } t,s \in \mathbb{T},
\;   \omega\in \Omega, \label{5}
\end{eqnarray}
where $id_X$ is the identity map on the state space $X$. We usually
say $\phi$ is a random dynamical system (over $\theta$).
 In this paper,  the time set $\mathbb{T}= \mathbb{Z} $ or
$\mathbb{R}$, and the state space $X$ is a locally compact complete
metric space (e.g., $\mathbb{R}^n$). It follows from (\ref{5}) that
$\phi(t,\omega)$ is a homeomorphism of $X$ and
 $$\phi(n,\omega)^{ -1 }=\phi(-n,\theta_n\omega).$$
 When $\phi$ is a discrete random dynamical system and $\varphi$
 is the time-one map of $\phi$, i.e. $\varphi(\omega)=\phi(1,\omega):X\rightarrow
 X$, then we call $\varphi$ the random homeomorphism determined by
 $\phi$. On the other hand, if $\varphi$ is a random
 homeomorphism, then it generates via iteration a discrete RDS
 $\phi(n,\omega,x)$.  So  we identify the discrete RDS $\phi$ as  a
 random homeomorphism.

 \medskip

The concept of  topological equivalence  of RDS is adapted from the
deterministic case  \cite{Arn, Chu}.  Let $\phi_1$ and $\phi_2$ be
two RDS over the same driving flow $ \theta$, but with state spaces
$X_1$ and $X_2$ respectively. The  RDS $\phi_1$ and $\phi_2$ are
said to be \textit{topologically  equivalent} if there exists a
mapping $\alpha : \Omega\times X_1
\rightarrow X_2 $ with the following properties:\\
 (i) The mapping $x \rightarrow \alpha(\omega,
x)$ is a homeomorphism from $X_1$ onto $X_2$ for   $\omega\in
\Omega$; \\
(ii) Both mappings $\omega \rightarrow\alpha(\omega, x_1)$ and
$\omega \rightarrow \alpha^{-1}(\omega, x_2)$ are measurable for
  $x_1 \in
X_1$ and $x_2 \in X_2$; \\
(iii) The cocycles $\phi_1$ and $\phi_2$ are cohomologous, i.e.,
\begin{equation}
\phi_2(t, \omega, \alpha(\omega, x)) =
\alpha(\theta_t\omega,\phi_1(t, \omega, x))\mbox{ for }    x \in X_1
\mbox{ and } \omega \in \Omega.
\end{equation}

\medskip
  We first recall some basic definitions from \cite{Liu}.
  A random compact set $N(\omega)$ is called a \textit{random isolating
  neighborhood} if it satisfies
  $$Inv N(\omega)\subset int N(\omega),$$
  where int $N(\omega)$ denotes the interior of $N(\omega)$ and
  $$Inv N(\omega)=\{x\in N(\omega)\mid \phi(n,\omega,x)\in N(\theta_n\omega),\forall n\in \mathbb{Z}\}. $$
    We call $S(\omega)$ a \textit{random isolated invariant set}
  if there exists a  random isolating neighborhood $N(\omega)$ such
  that $S(\omega)=Inv N(\omega)$. A random compact set $N(\omega)$ is called a
 \textit{random isolating block} if it satisfies
  $\varphi(\theta_{-1}\omega,N(\theta_{-1}\omega))\bigcap N(\omega)\bigcap \varphi^{-1}(\theta\omega,N(\theta\omega))$
  $\subset int N(\omega)$.

  For a random set $N(\omega)$, we define its \textit{exit set}
to be $N^{-}(\omega):=\{x\in N(\omega)\mid \varphi(\omega,x)\notin
int N(\theta\omega)\}$.

We now   introduce the concept of a \textit{random filtration pair},
as the random counterpart of the deterministic concept \cite{Fra}.

\begin{definition} [Random filtration pair]
Assume that $N(\omega)$ is a random isolating neighborhood,
$L(\omega)\subset N(\omega)$ is a random compact set and $S(\omega)$
is a random isolated invariant set inside $N(\omega)$. Assume
further that $N(\omega)=cl(int N(\omega))$ and $L(\omega)=cl(int
L(\omega))$. We call $(N(\omega),L(\omega))$ a random filtration
pair for
$S(\omega)$ if the following conditions hold:\\
 (i) $cl(N(\omega)\setminus L(\omega))$ is a random isolating neighborhood of
 $S(\omega)$;\\
 (ii) $L(\omega)$ is a random neighborhood of $N^{-}(\omega)$ in
 $N(\omega)$; and \\
 (iii) $\varphi(\omega,L(\omega))\bigcap
 cl(N(\theta\omega) \backslash L(\theta\omega))=\emptyset$.
\end{definition}

Assume that $P = (N(\omega),L(\omega))$ is a random filtration pair
for $S(\omega)$.  Let $N_L(\omega)$ be the random quotient space(or
random pointed space) $N(\omega)/L(\omega)$, where
$N_L(\omega)=(N(\omega) \setminus L(\omega)\bigcup [L(\omega)],
[L(\omega)])$ for all $\omega\in \Omega$. If $L(\omega)=\emptyset$,
then $N_L(\omega)=N(\omega)\bigcup [\emptyset]$. A map
$\varphi_P(\omega,\cdot):N_L(\omega)\rightarrow N_L(\theta\omega)$
is called a random pointed space map associated to $P$ if
\begin{displaymath}
\varphi_P(\omega,\cdot)= \left\{ \begin{array}{ll}
[L(\theta\omega)], & x=[L(\omega)] \quad\textrm{or}\quad  \varphi(\omega,x)\notin N(\theta\omega),\\
p(\theta\omega,\varphi(\omega,x)),& \textrm{otherwise,}\\
\end{array} \right.
\end{displaymath}
where $p(\omega,\cdot): N(\omega)\rightarrow N_L(\omega)$ is the
random quotient map. Assume that $C$ and $D$ are two random pointed
spaces, and $c$ and $d$ are two \textit{Caratheodory functions}
(i.e., $c(\omega,\cdot)$, $d(\omega,\cdot)$
  are continuous and $c(\cdot,x)$, $d(\cdot,x)$ are measurable)
defined as
\begin{equation*}
c(\omega,\cdot):C(\omega)\rightarrow C(\theta \omega),\quad
d(\omega,\cdot):D(\omega)\rightarrow D(\theta \omega).
\end{equation*} Assume that $c$ and $d$ preserve base-points.
Two random pointed spaces $(C, c)$ and $(D, d)$ are called\textit{
random shift equivalent} and denoted by $(C, c) \sim  (D, d)$, if
there exist random maps $r(\omega,\cdot):C(\omega) \rightarrow
D(\theta_{n_1} \omega), s(\omega,\cdot): D(\omega)\longrightarrow
C(\theta_{n_2}\omega)$ with measurable $n_1 = n_1(\omega)$ and $n_2
= n_2(\omega)$ such that they preserve base points and the following
diagrams are quasi-commutative:
\begin{eqnarray}\label{2.4}
\begin{CD}
 C(\omega) @>{\rm c(\omega,\cdot)}>> C(\theta\omega) \\
@V{r(\omega,\cdot)}V V@VV{r(\theta\omega,\cdot)}V \\
D(\theta_1\omega ) @>>{\rm d({\theta_{n_1}}(\omega)},\cdot)>
D(\theta_\ast \omega)
\end{CD}
\end{eqnarray}
\begin{eqnarray}\label{2.5}
\begin{CD}
 D(\omega) @>{\rm d(\omega,\cdot)}>> D(\theta\omega) \\
@V{s(\omega,\cdot)}V V@VV{s(\theta\omega,\cdot)}V \\
C(\theta_2\omega ) @>>{\rm c({\theta_{n_2}}(\omega),\cdot)}>
C(\theta_\ast \omega)
\end{CD}
\end{eqnarray}
and the following holds:
\begin{eqnarray}
  r(\theta_{n_2}\omega,s(\omega,\cdot))=d^{n_2(\omega)+n_1(\theta_{n_2}\omega)}(\omega,\cdot), \label{rs}\\
 s(\theta_{n_1}\omega,r(\omega,\cdot))=c^{n_1(\omega)+n_2(\theta_{n_1}\omega)}(\omega,\cdot).  \label{sr}
\end{eqnarray}

 Here  ``quasi-commutative'' in the diagram \eqref{2.4}  is in the sense
that
\begin{eqnarray}\label{2.3}
\end{eqnarray}
\begin{eqnarray*} \left\{ \begin{array}{l}
r(\theta\omega,c(\omega,\cdot))=d^{n_1(\theta\omega)-n_1(\omega)}(\theta_{n_1(\omega)+1}\omega,
d(\theta_{n_1(\omega)}\omega,r(\omega,\cdot))), n_1(\theta\omega)\geq n_1(\omega),\\
d^{n_1(\omega)-n_1(\theta\omega)}(\theta_{n_1(\theta\omega)+1}\omega,r(\theta\omega,c(\omega,\cdot)))=d(\theta_{n_1(\omega)}\omega,r(\omega,\cdot)),
 n_1(\theta\omega)< n_1(\omega).\\
 \end{array} \right.
\end{eqnarray*}
In short, equations \eqref{2.3} is written as
\begin{eqnarray*}
r \circ  c = d^\bigtriangleup \circ d \circ r, \quad
d^\bigtriangleup \circ r \circ c = d \circ r,
\end{eqnarray*}
where $\bigtriangleup$ denote the adjustment that is understood as
in \eqref{2.3}. Similarly we interpret the diagram in \eqref{2.5}.
Moreover, the \textit{subscript} $_*$ in \eqref{2.4} and \eqref{2.5}
denotes $\max \{n_1(\omega)+1,n_1(\theta\omega)+1\}$ and $\max
\{n_2(\omega)+1,n_2(\theta\omega)+1\}$, respectively.


It can be proved that the shift equivalence is an equivalence
relation. If $P=(N(\omega),L(\omega))$ and
$P'=(N'(\omega),L'(\omega))$ are two random filtration pairs for
$S(\omega)$, then the induced random maps, $\phi_P$ and $\phi_{P'}$
on the corresponding random pointed spaces, are random shift
equivalent.

 Let $f$ and $g$ be   Caratheodory functions. Then  $f$ is called random
homotopic to $g$, denoted by $f\simeq  g$, if there exists a map $H
: [0, 1]\times \Omega\times C(\omega)\rightarrow C(\theta_n\omega)$
such that $H(\cdot,\omega,\cdot)$ is continuous, $H(\lambda, \cdot,
x) $ is measurable, and
\begin{displaymath}
\left\{ \begin{array}{ll}
H(0,\cdot,\cdot)=f(\cdot,\cdot),\\
H(1,\cdot,\cdot)=g(\cdot,\cdot).\\
\end{array} \right.
\end{displaymath}
Denote $[f]$ the random homotopy class with $f$ the representative
element. The random homotopy equivalence classes $(C, [c])$ and $(D,
[d])$ are called random shift equivalent if there exist random
homotopy classes $[r(\omega)]: C(\omega) \rightarrow
D(\theta_{n_1}\omega)$ and $[s(\omega)]: D(\omega)\rightarrow
C(\theta_{n_2}\omega)$, where $n_i=n_i(\omega) (i=1,2)$ are
measurable, such that
\begin{eqnarray*}
 [r\circ c]=[d^\bigtriangleup \circ d \circ r],\quad [s\circ
d]=[c^\bigtriangleup \circ c\circ s],\quad
 [r\circ s]=[d^*],\quad
[s\circ r]=[c^*],
\end{eqnarray*}
where the \textit{superscript} $^*$  in the latter two equations
denotes $n_2(\omega)+n_1(\theta_{n_2}\omega)$ and
$n_1(\omega)+n_2(\theta_{n_1}\omega)$, respectively (see \eqref{rs}
and \eqref{sr}).

\begin{definition}[Random Conley index \cite{Liu}]
  Assume that $\varphi$ is the time-one map of a discrete random dynamical system, $S(\omega)$ is a
random isolated invariant set for $\varphi$ and $P =
(N(\omega),L(\omega))$ is a random filtration pair for $S(\omega)$.
Let $h_P (S(\omega),\varphi)$ be the random homotopy class
$[\varphi_P]$ on the random pointed space $N_L(\omega)$ with
$\varphi_P$ a representative element. Then   the random shift
equivalent class of $h_P(S(\omega),\varphi)$, denoted by
$h(S(\omega),\varphi)$, is defined as the random Conley index for
$S(\omega)$.
\end{definition}

 We present a few
  properties of the random Conley index. These will be needed
  in the following sections.

\begin{lemma}\label{2.1}
  If $S_1(\omega)$ and $S_2(\omega)$ are disjoint random isolated invariant
  sets, then $S(\omega)=S_1(\omega)\bigcup S_2(\omega)$ is a random
  isolated invariant set.
\end{lemma}
\begin{proof}
Since $S_1(\omega)$ and $S_2(\omega)$ are random isolated invariant
sets, there exist disjoint random isolating blocks $N_1(\omega)$ and
$N_2(\omega)$ such that $S_1(\omega)=Inv N_1(\omega)$ and
$S_2(\omega)=Inv N_2(\omega)$. Let
$N'_1(\omega):=\varphi(\theta_{-1}\omega,N_1(\theta_{-1}\omega))\bigcap
N_1(\omega)\bigcap \varphi^{-1}(\theta\omega,N_1(\theta\omega))$ and
$N'_2(\omega):=\varphi(\theta_{-1}\omega,N_2(\theta_{-1}\omega))\bigcap
N_2(\omega)\bigcap \varphi^{-1}(\theta\omega,N_2(\theta\omega))$.
Then $N'_1(\omega)$ and $N'_2(\omega)$ are disjoint random isolating
neighborhoods by the invariance of $S_1(\omega)$ and $S_2(\omega)$.
We can check that $N(\omega):=N_1'(\omega)\bigcup N_2'(\omega)$ is a
random isolating neighborhood for $S(\omega)$.
\end{proof}

\begin{lemma}\label{2.2}
  If the random dynamical systems $\phi_1$ and $\phi_2$ are
  topologically equivalent and $S(\omega)$ is a random isolated invariant
  set with respect to $\phi_1$, then $ \alpha(\omega,S(\omega))$
  is a random isolated invariant set with respect to $\phi_2$ and $$h(S(\omega),\phi_1)=
  h(\alpha(\omega,S(\omega)),\phi_2).$$
\end{lemma}
\begin{proof}    Since $S(\omega)$ is a random isolated invariant
  set with respect to $\phi_1$, there exists a random isolating
  neighborhood $N(\omega)$ such that $S(\omega)=Inv N(\omega)$.
  Note
    that $\alpha(\omega,N(\omega))$ is a random isolating
    neighborhood. We now show that
    $\alpha(\omega,S(\omega))=Inv\alpha(\omega, N(\omega))$.
   In fact, if $x\in
  S(\omega)$, by definition, we have $\phi_1(n,\omega,x)\in
  N(\theta_n\omega)$, $\forall n \in \mathbb{Z}$. So $\alpha(\theta_n\omega,\phi_1(n,
  \omega,x))\in \alpha(\theta_n\omega,N(\theta_n\omega))$, $\forall n \in
  \mathbb{Z}$. That is $$\phi_2(n,\omega,\alpha(\omega,x))\in
  \alpha(\theta_n\omega,N(\theta_n\omega)).$$ We have $\alpha(\omega,S(\omega))\subset Inv\alpha(\omega, N(\omega))$.
 Similarly we can show the opposite inclusion.

  Suppose $P=(N(\omega),L(\omega))$ is a random filtration pair for
  $S(\omega)$. We check that
  $P'=(\alpha(\omega,N(\omega)),\alpha(\omega,L(\omega)))$ is a
  random filtration pair for $\alpha(\omega,S(\omega))$
  as follows. (i) If $x\in Inv \alpha(\omega,cl(N(\omega)\setminus
  L(\omega)))$, then there exists $y\in cl(N(\omega)\setminus
  L(\omega))$ such that $x=\alpha(\omega,y)$ and $\phi_2(n,\omega,x)=\alpha(\theta_n\omega,\phi_1(n,\omega,y))\in \alpha(\theta_n\omega,
  cl(N(\theta_n\omega)\setminus L(\theta_n\omega)))$, and
  we have $\phi_1(n,\omega,y)\in cl(N(\theta_n\omega)\setminus
  L(\theta_n\omega))$.
  Thus $y\in Inv cl(N(\omega)\setminus L(\omega))$. Using the fact that $cl(N(\omega)\setminus
  L(\omega))$ is a random isolating neighborhood of $S(\omega)$, we
  have $y\in S(\omega)$. Hence $\alpha(\omega,y)\in \alpha(\omega,
  S(\omega))$. It follows that $$Inv \alpha(\omega,cl(N(\omega)\setminus L(\omega)))\subset
  \alpha(\omega,S(\omega)).$$ Similarly we can prove $
  \alpha(\omega,S(\omega))\subset Inv \alpha(\omega,cl(N(\omega)\setminus
  L(\omega)))$. We conclude that $\alpha(\omega,cl(N(\omega)\setminus L(\omega)))$
  is a random isolating neighborhood of $\alpha(\omega,S(\omega))$.
  (ii) Denote $\alpha^-(\omega,N(\omega))$ the exit set of
  $\alpha(\omega,N(\omega))$. Then
  $\alpha^-(\omega,N(\omega))=\alpha(\omega,N^-(\omega))$.
It is easy to see that  $\alpha(\omega,L(\omega))$ is a random
  isolating
  neighborhood of $\alpha^-(\omega,N(\omega))$. (iii) Note that $$\alpha(\theta\omega,\varphi_1(\omega,L(\omega)))
  =\varphi_2(\omega,\alpha(\omega,L(\omega)))$$ and
   $$\varphi_1(\omega,L(\omega))\bigcap cl(N(\theta\omega)\backslash L(\theta\omega
  ))=\emptyset. $$So $$ \alpha(\theta\omega,\varphi_1(\omega,L(\omega)))\bigcap \alpha(\theta\omega, cl(N(\theta\omega)\backslash
  L(\theta\omega)))=\emptyset, $$i.e., $ \varphi_2(\omega,\alpha(\omega,L(\omega)))\bigcap cl(\alpha(\theta\omega, N(\theta\omega))
  \backslash
 \alpha(\theta\omega, L(\theta\omega)))=\emptyset$.

Denote $\alpha_L(\omega,N(\omega))$ the random quotient space
$\alpha(\omega,N(\omega))/ \alpha(\omega,L(\omega))$. Let
\begin{eqnarray*}
\varphi_{2 P'}(\omega,\cdot)=\left\{ \begin{array}{ll}
[\alpha(\theta\omega,L(\theta\omega))],& x=[\alpha(\omega,L(\omega))]  \mbox{ or }  \varphi_2(\omega,x)\notin \alpha(\theta\omega,N(\theta\omega)),\\
p_2(\theta\omega,\varphi_2(\omega,x)),
&\mbox{otherwise},\\
\end{array} \right.
\end{eqnarray*}
where $p_2(\omega,\cdot):\alpha(\omega,N(\omega))\rightarrow
\alpha_L(\omega,N(\omega))$ is the quotient map.\\

Consider the following two maps:
$r(\omega,\cdot):N_L(\omega)\rightarrow\alpha_L(\omega,N(\omega)),$
\begin{eqnarray*}
 r(\omega,\cdot)=\left\{ \begin{array}{ll}
[\alpha(\omega,L(\omega))],& x=[L(\omega)],\\
p_2(\omega,\alpha(\omega,x)),
&\mbox{otherwise},\\
\end{array} \right.
\end{eqnarray*}
and
 $s(\omega,\cdot):\alpha_L(\omega,N(\omega))\rightarrow
N_L(\omega),$
\begin{eqnarray*}
s(\omega,\cdot)=\left\{ \begin{array}{ll}
[L(\omega)],& x=[\alpha(\omega,L(\omega))],\\
p(\omega,\alpha^{-1}(\omega,x)),
&\mbox{otherwise}.\\
\end{array} \right.
\end{eqnarray*}
We verify that $r(\omega,\cdot)$ and $s(\omega,\cdot)$ are
continuous, and $r(\cdot,x)$ and $s(\cdot,x)$ are measurable.
Moreover we have
\begin{eqnarray*}
  [r\circ \varphi_{1 P}]=[\varphi_{2 P'}\circ
 r], [s\circ \varphi_{2 P'}]=[ \varphi_{1 P}\circ
  s],
  [r\circ s]=[I], [s\circ r]=[I].
\end{eqnarray*}
Thus by the definition of random Conley index we conclude that
$$h(S(\omega),\phi_1)=
  h(\alpha(\omega,S(\omega)),\phi_2).$$
\end{proof}

\section{Prime random isolated invariant sets}

  We call $S(\omega)$ a \textit{prime random isolated invariant set} for $\varphi$ if
  $S(\omega)$ is a random  isolated invariant set for $\varphi$, and for
  any proper subset $S'(\omega)\subset S(\omega)$, $S'(\omega)\neq
  \emptyset$, $S'(\omega)$ is not a random isolated invariant set
  for $\varphi$. Suppose $\varphi$ has only finitely many prime random isolated
  invariant sets $S_1(\omega),S_2(\omega),\cdots, $ and $S_r(\omega)$, then
   $T_f(\omega)=\bigcup\limits_{i=1}^r S_i(\omega)$  is called an
 \textit{ extreme maximal random isolated invariant set }for $\varphi$.
For the deterministic case see \cite{Fu}.

\begin{lemma}\label{3.2}
  For any two different prime random isolated invariant sets $S_1(\omega)$ and $S_2(\omega)$,
we have $S_1(\omega)\bigcap S_2(\omega)=\emptyset$ for $\omega\in
\Omega$.
\end{lemma}

\begin{proof}
Assume  that $N_1(\omega)$ and $N_2(\omega)$ are random isolating
neighborhoods for $S_1(\omega)$ and $S_2(\omega)$ respectively. We
set $S(\omega)=S_1(\omega)\bigcap S_2(\omega)$ and $N(\omega) =
N_1(\omega)\bigcap N_2(\omega)$. If $ S(\omega)\neq \emptyset$ then
there exists $x$ such that $x\in S(\omega)$. We have $x \in
S_i(\omega) (i = 1, 2)$, which imply $\varphi^n(\omega,x) \in
S_1(\theta_n\omega) \bigcap S_2(\theta_n\omega)=S(\theta_n\omega),
n\in \mathbb{Z}$.  Since $S(\omega)$ is compact and forward
invariant, by the fact that $InvS(\omega)=\Omega_S(\omega)$, we have
\begin{equation*}
InvS(\omega)\neq \emptyset.
\end{equation*}
It is clear  that $InvS(\omega) \subset InvN(\omega)$. On the other
hand, for any $x \in InvN(\omega)$, since $N(\omega) =
N_1(\omega)\bigcap N_2(\omega)$, we have $\varphi^n(\omega,x) \in
N_1(\theta_n\omega) \bigcap N_2(\theta_n\omega), n\in \mathbb{Z}$,
which implies $\varphi^n(\omega,x) \in Inv
N_i(\theta_n\omega)=S_i(\theta_n\omega) (i=1,2)$. It follows that
  $x\in InvS(\omega)$. Furthermore, we have
\begin{eqnarray*}
&&InvS(\omega) = Inv(S_1(\omega)\bigcap S_2(\omega)) \subset
S_1(\omega) \bigcap S_2(\omega) \subset intN_1(\omega) \bigcap
intN_2(\omega)\\&&=int(N_1(\omega)\bigcap N_2(\omega)) =
intN(\omega).
\end{eqnarray*}
Thus $InvS(\omega)$ is a nonempty random isolated invariant (with
respect to $\varphi$) proper subset of $S_i(\omega) (i = 1, 2)$,
which is a contradiction. So $S(\omega) = \emptyset$, i.e.,
$S_1(\omega)$ and $S_2(\omega)$ are disjoint.
\end{proof}

 By Lemmas \ref{3.2} and   \ref{2.1}, we
see that an extreme maximal random isolated invariant set is indeed
a random isolated invariant set. As in \cite{Fu}, if there exists no
nonempty prime random isolated invariant set or if there are
infinitely many prime random isolated invariant sets, then we define
the extreme maximal random isolated invariant set for $\varphi$ by
$T_\varphi(\omega)= \emptyset$. Thus, for every random homeomorphism
$\varphi$, there exists a unique extreme maximal random random
isolated invariant set.

\section{Main result} \label{main}

Before   presenting our main result, we recall the definition for
abstract bifurcation points\cite{Arn}.

\begin{definition} [Abstract bifurcation point]
  Let $ \varphi_\lambda $ be a family of  RDS
  on $X$, parameterized by $\lambda\in R^k$. A parameter value $\lambda_0$ is called an abstract
  bifurcation point of the family if the family is not structurally
  stable at $\lambda_0$, i.e., if in any neighborhood of $\lambda_0$
  there is a parameter value  $\lambda$ such that $\varphi_\lambda$ and
  $\varphi_{\lambda_0}$ are not topologically equivalent.
\end{definition}

For the extreme maximal random isolated invariant set
$T_{\varphi_\lambda}(\omega)$, we use
$M(T_{\varphi_\lambda})(\omega)$ to denote the cardinal number of
the set $\{S(\omega)\mid S(\omega)$ is a prime random isolated
invariant set for $\varphi_\lambda\}$. If
$T_{\varphi_\lambda}(\omega)=\emptyset$, we set
$M(T_{\varphi_\lambda})(\omega)=0$.  The main result of this paper
is based on the  following Lemma.

\begin{lemma}\label{3.1}
  Suppose $\varphi_\lambda $ and $\varphi_{{\lambda_0}}$ are topologically equivalent
   under $\alpha$. Let $T_{\varphi_\lambda}(\omega)$ and $T_{\varphi_{\lambda_0}}(\omega)$
    be the extreme maximal random isolated
invariant sets for $\varphi_\lambda $ and $\varphi_{\lambda_0}$
respectively. Then
$M(T_{\varphi_\lambda})(\omega)=M(T_{\varphi_{\lambda_0}})(\omega)$
and $\alpha(\omega,
T_{\varphi_\lambda}(\omega))=T_{\varphi_{\lambda_0}}(\omega)$.
\end{lemma}

\begin{proof}
  We first consider the   case when   $T_{\varphi_\lambda}(\omega)\neq \emptyset$. Then there exists
  an integer $r > 0$ such that $T_{\varphi_\lambda}(\omega)=\bigcup
  \limits_{i=1}^r S_i(\omega)$, where each $S_i(\omega)$ is a prime random isolated
  invariant set for $\varphi_\lambda $.  Since $\varphi_\lambda $ and $\varphi_{\lambda_0}$
 are topologically equivalent, by Lemma \ref{2.2}, for each $i_0 \in
\{1,\cdots,r\}$, $\alpha(\omega,S_{i_0}(\omega))$ is an isolated
invariant set of $\varphi_{\lambda_0}$. Suppose
$\alpha(\omega,S_{i_0}(\omega))$ is not prime random isolated set,
i.e., there exists nonempty proper subset $S_{i_0}'(\omega) \subset
\alpha(\omega,S_{i_0}(\omega))$ such that $S_{i_0}'(\omega)$ is a
random isolated invariant set. Then
$\alpha^{-1}(\omega,S_{i_0}'(\omega)) \subset S_{i_0}(\omega) $ is a
nonempty random isolated invariant set of $\varphi_\lambda$, which
contradicts with the fact that $S_{i_0}(\omega)$ is  a prime random
isolated invariant set. Thus $\alpha$ maps a prime random isolated
invariant set for  $\varphi_\lambda$ to a prime random isolated
invariant set for  $\varphi_{\lambda_0}$. Similarly, $\alpha^{-1}$
maps a prime random isolated invariant set for $\varphi_{\lambda_0}$
to a prime random isolated invariant set for $\varphi_\lambda$.
Therefore,
$M(T_{\varphi_\lambda})(\omega)=M(T_{\varphi_{\lambda_0}})(\omega)$
and $\alpha(\omega, T_{\varphi_\lambda}(\omega))=\alpha(\omega,
\bigcup
  \limits_{i=1}^r S_i(\omega))=\bigcup \limits_{i=1}^r\alpha(\omega,
 S_i(\omega))
  =T_{\varphi_{\lambda_0}}(\omega)$. When
  $T_{\varphi_\lambda}(\omega)=\emptyset$,
  we have $T_{\varphi_{\lambda_0}}(\omega)=\emptyset$. So we still have $\alpha(\omega, T_{\varphi_\lambda}
  (\omega))=T_{\varphi_{\lambda_0}}(\omega)$.
\end{proof}

This is our main result on bifurcation.
\begin{theorem}
[Bifurcation points for discrete-time  random dynamical
system]\label{3.4} Consider     a family of discrete-time RDS $
\varphi_\lambda $, parameterized by $\lambda\in R^k$. If for any
neighborhood $U$ containing $\lambda_0\in R^k$ there exists $\lambda
\in U$ such that $M(T_{\varphi_\lambda})(\omega)\neq
M(T_{\varphi_{\lambda_0}})(\omega)$ or
$h(T_{\varphi_{\lambda}}(\omega), \varphi_{\lambda})\neq
h(T_{\varphi_{\lambda_0}}(\omega), \varphi_{\lambda_0})$, then
$\lambda_0$ is an abstract bifurcation point of $\varphi_\lambda$.
\end{theorem}

\begin{proof}
Suppose $\lambda_0$ is not an abstract bifurcation point, then there
exists a neighborhood $U$ containing $\lambda_0$ such that for each
$\lambda\in U$, there exists a homeomorphism $\alpha_\lambda$ on $X$
such that $\varphi_\lambda$ and
  $\varphi_{\lambda_0}$ are  topologically equivalent.

  From Lemma \ref{3.1}, we have $M(T_{\varphi_\lambda})(\omega)=
  M(T_{\varphi_{\lambda_0}})(\omega)$ and
$\alpha_\lambda(\omega, T_{\varphi_{\lambda}}(\omega))=
T_{\varphi_{\lambda_0}}(\omega)$. By Lemma \ref{2.2}, we have
$h(T_{\varphi_{\lambda}}(\omega), \varphi_{\lambda})= $
$h(T_{\varphi_{\lambda_0}}(\omega), \varphi_{\lambda_0})$, a
contradiction. So $\lambda_0$ is an abstract bifurcation point of
$\varphi_\lambda$.
\end{proof}

Although the Conley index for continuous-time RDS (i.e., the time
set $\mathbb{T}= \mathbb{R}$) is not available at present, we can
still apply   the Conley index for discrete-time RDS to detect
abstract bifurcation points of continuous-time  RDS. Namely we have
the following bifurcation result.

\begin{coro}
[Bifurcation points for continuous-time  RDS] \label{bifn2}
  Consider     a family of continuous-time RDS $ \varphi_\lambda $,
parameterized by $\lambda\in R^k$. If  for any neighborhood $U$
containing $\lambda_0\in R^k$ there exists $\lambda \in U$ such that
$M(T_{\varphi_\lambda})(\omega)\neq
M(T_{\varphi_{\lambda_0}})(\omega)$ or
$h(T_{\varphi_\lambda}(\omega), \varphi_\lambda)\neq
h(T_{\varphi_{\lambda_0}}(\omega), \varphi_{\lambda_0})$ for the
corresponding discrete-time RDS $\varphi_\lambda(n,\omega,$ $x)$,
$n\in \mathbb{Z}$, then $\lambda_0$ is an abstract  bifurcation
point for the original family of continuous-time RDS
$\varphi_\lambda$.
\end{coro}

\begin{proof}
Note that the discrete-time RDS $\varphi_\lambda(n,\omega,x) $ is
generated via iteration by the time-one map of the continuous-time
RDS $\varphi_\lambda$. This result thus follows from  Theorem
\ref{3.4} and the definition of abstract bifurcation point.
\end{proof}

We present two   examples to demonstrate the above bifurcation
results for RDS.

\begin{example}[Bifrucation in a discrete-time RDS]
 Consider a family of discrete-time RDS $\varphi_\lambda$ on  $\mathbb{R}$,
\begin{displaymath}
\varphi_\lambda(\omega,x) = \left\{ \begin{array}{ll}
x+x^2+\lambda\xi(\omega),& x\geq -\frac 1 2, \\
-\frac 1 2x+\lambda \xi(\omega), & x<-\frac 1 2.\\
\end{array} \right.
\end{displaymath}
 where $\lambda \in \mathbb{R} $ is a real parameter, and $ \xi(\omega)$ is a given positive
 random variable. There exists a prime random isolated
invariant set  for $\lambda \leq  0$. But for $\lambda>0$, there is
no  such a prime random isolated invariant set. So $\lambda = 0 $ is
an abstract bifurcation point for the discrete-time RDS
$\varphi_\lambda$ by Theorem \ref{3.4}.
\end{example}

We revise an example from \cite{Smo} to fit our purpose here.
\begin{example} [Bifrucation in a continuous-time RDS]
Consider a  family of scalar random differential equations
\begin{equation}\label{3.5}
  x'=f_\lambda(x,\theta_t\omega),
\end{equation}
parameterized by a real parameter $\lambda \in [-1, 1]$. Assume that
\eqref{3.5} generates a family of RDS $\varphi_\lambda (t, \omega,
x)$ and that there exists a unique fixed point $0$, $|\lambda|\leq
1$. Denote the corresponding discrete-time RDS by
$\varphi_\lambda(n,\omega,x)$. From Theorem 7.2 of  \cite{Liu} we
know that $\{0\}$ is the random isolated invariant set of
$\varphi_\lambda(n,\omega,x)$. If
$\lim\limits_{n\rightarrow\pm\infty}{\varphi_\lambda(n,\omega,x)}=0$
for $\lambda>0$, $x\in \mathbb{R}$, and $\lim\limits_{n\rightarrow
+\infty}{\varphi_\lambda(n,\omega,x)}=0$, $\lim\limits_{n\rightarrow
-\infty}{\varphi_\lambda(n,\omega,x)}= -\infty$ for $\lambda<0$,
$x\in \mathbb{R}$. Let $\b{0}$ be the random Conley index of random
pointed spaces consisting of just one random point with random
constant maps as their corresponding random pointed space maps. Then
the Conley index of $\{0\}$ for $\lambda>0$ is not equal to $\b{0}$,
but the Conley index of $\{0\}$ for $\lambda<0$ is $\b{0}$. We can
conclude that $ \lambda=0$ is the abstract bifurcation point of
continuous-time RDS $\varphi_\lambda$ by Corollary \ref{bifn2}.
\end{example}

\bibliographystyle{amsplain}

\end{document}